\documentclass[12pt]{amsart}
\usepackage{graphicx}

\usepackage{amssymb} 

\usepackage[color,matrix,arrow]{xy}
\usepackage{tikz}
\usetikzlibrary{decorations.markings,arrows,automata,arrows,backgrounds,snakes}
\usepackage{graphicx}

\usepackage{subcaption} 

\usepackage[numbers]{natbib} 

\usepackage{hyperref} 

\hypersetup{bookmarks=true,
	unicode=true,
	colorlinks=true,
	citecolor=black,
	linkcolor=black,
	urlcolor=black,
	plainpages=false,
	pdfpagelabels=true}

\usepackage{tikz-cd}
\usepackage{natbib}

\usepackage{pgf}

 \usepackage{url}	
 \allowdisplaybreaks 

\newtheorem{teo}{Theorem}[subsection]
\newtheorem{theorem}[teo]{Theorem}
\newtheorem*{thm*}{Theorem}
\newtheorem{cor}[teo]{Corollary}
\newtheorem{corollary}[teo]{Corollary}

\newtheorem{lema}[teo]{Lemma}
\newtheorem{lemma}[teo]{Lemma}
\newtheorem{prop}[teo]{Proposition}

\newtheorem*{claim*}{Claim}

\theoremstyle{definition}
\newtheorem{definition}[teo]{Definition}
\newtheorem{example}[teo]{Example}

\theoremstyle{remark}
\newtheorem{remark}{Remark}

\numberwithin{figure}{section}

\newcommand{\Id}{\mathrm{id}} 

\newcommand{\Imm}{\mathrm{Im}}
\newcommand{\rank}{\mathrm{rank \,}}

\newcommand{\co}{\colon}
\newcommand{\R}{\mathbb{R}}

\newcommand{\Z}{\mathbb{Z}}

\newcommand{\cat}{\mathrm{cat}}
\newcommand{\crit}{\mathrm{crit}}

\newcommand{\sd}{\mathrm{sd}}

\newcommand{\HCcat}{\mathrm{hccat}}

\begin{document}
\title[Morse-Bott Theory on posets]{Morse-Bott Theory on posets and an homological Lusternik-Schnirelmann Theorem}
\thanks{The first and the fourth authors were partially supported by MINECO Spain Research Project MTM2015--65397--P and Junta de Andaluc\'ia Research Groups FQM--326 and FQM--189. The second author was partially supported by MINECO-FEDER research project MTM2016--78647--P. The third author was partly supported by Ministerio de Ciencia, Innovaci\'on y Universidades,  grant FPU17/03443. The second and third authors were partially supported by Xunta de Galicia ED431C
	2019/10 with FEDER funds.}
\author[Fern\'{a}ndez-Ternero, Mac\'{i}as-Virg\'os, Mosquera-Lois, Vilches]{D. Fern\'{a}ndez-Ternero$^{(1)}$, E. Mac\'{i}as-Virg\'os$^{(2)}$, \\
D. Mosquera-Lois$^{(2)}$, J.A. Vilches$^{(1)}$}
\address{$^{(1)}$Departamento de Geometr\'ia y Topolog\'ia, Universidad de Sevilla, Spain. 
	\\
$^{(2)}$ Instituto de Matemáticas, Universidade de Santiago de Compostela, Spain.}
\email{desamfer@us.es, quique.macias@usc.es, david.mosquera.lois@usc.es, vilches@us.es}

\begin{abstract} We develop Morse-Bott theory on posets, generalizing both discrete Morse-Bott theory for regular complexes and Morse theory on posets. Moreover, we prove a Lusternik-Schnirelmann theorem for general matchings on posets, in particular, for Morse-Bott functions.
\end{abstract}


\maketitle
\setcounter{tocdepth}{1}
\tableofcontents

\section{Introduction}\label{INTRO}

 Since its introduction, Morse Theory has been an active field of research and connections with many different areas of Mathematics have been found. That interaction has led to several adaptations of Morse Theory to different contexts, for example: PL versions by Banchoff \cite{Banchoff,Banchoff2} and by Bestvina and Brady \cite{Bestvina} and a purely combinatorial approach by Forman \cite{Forman,Forman2}. Nowadays, not only pure mathematics benefit from that interaction, but also applied mathematics \cite{Ghristbook} due to the importance of discrete settings.

Roughly speaking, Morse Theory addresses the study of the topology (homology, originally) of a space by breaking it into ``elementary'' pieces. That is achieved by the so called Fundamental or Structural Theorems of Morse Theory, which assert that the object of study (for example a smooth manifold or a simplicial complex) has the homotopy type of a CW-complex with a given cell structure determined by the criticality of a Morse function defined on it \cite{Forman2,Milnor}.

Originally Morse Theory began with the definition of Morse function itself, that is, a smooth function with non-degenerate critical points so, the only  critical objects allowed were points. That was overcame by Morse-Bott theory \cite{Bott2}, which broadened the class of critical objects by including non-degenerate critical submanifolds. To avoid the critical objects must be non-degenerate, this led to the introduction of Lusternik-Schnirelmann theory \cite{LUSTERNIK_viejo}.

In the context of Morse Theory, Morse inequalities guarantee that the number of critical points of a Morse function $f\co X\to \R$ is an upper bound for the homological complexity of the space $X$. The role of the Morse inequalities in the setting of Lusternik-Schnirelmann theory is played by the so called Lusternik-Schnirelmann theorem, which asserts that the weighted sum of the number of critical objects is an upper bound for the category of the space \cite{LUSTERNIK_viejo}.

Recent works have shown that it is possible to approach important problems regarding posets by using topological methods. See for example Barmak and Minian's work on the realizability of groups as the automorphism groups of certain posets \cite{BARMAK_posets} or Stong's work on groups on the way the homotopy type of the poset of  non-trivial $p$-subgroups  ordered by inclusion determines algebraic properties of the group  \cite{Stong_groups}.  Moreover, it is expected that recent discrete analogues of some classic concepts from differential topology shed light on their originals counterparts \cite{Forman_6}. Therefore, it makes sense to study the topology of finite spaces by means of some adapted version of Morse theory to this context.

An invariant of a space $X$ is the smallest number of open sets that cover $X$ and that satisfy certain properties, such as: being elementary in a certain sense, for instance acyclic or contractible (see \cite{Fox,Larranhaga} for more examples). More generally, and analogously, a categorical invariant of an object $X$ (such as a simplicial complex) can be defined as the smallest number of subobjects needed to cover $X$ and that verify certain properties (see for example \cite{QuiMos2,SamQuiErJa,SamQuiErJa2,Tanaka3}). In vague terms, an invariant provides a certain measure for the complexity of an object. For example, the Lusternik-Schnirelmann category measures, in a particular manner, how far is a space from being contractible. 

This work addresses two aims. First, to develop Morse-Bott theory in the context of finite spaces, generalizing both Morse theory for posets introduced by Minian \cite{Minian} and discrete Morse-Bott theory \cite{Forman3}. In particular, we prove an integration result for matchings, the Fundamental Theorems of Morse-Bott theory in this setting and several generalization of Morse inequalities for arbitrary matchings. Second, we introduce an homological category and we prove the Lusternik-Schnirelmann theorem. 

We describe the main motivations for the latter goal.  First, the absence of a discrete Lusternik-Schnirelmann theorem for arbitrary matchings, not even in the simplicial setting, and second, the gap in the literature even for a Lusternik-Schnirelmann theorem for Morse (acyclic) matchings. Nevertheless, several attempts were made. On the one side, a subset of the authors in joint work with Scoville proved a Lusternik-Schnirelmann theorem for a notion of simplicial  category and acyclic matchings in the context of simplicial complexes \cite{SamQuiNiJa}. However, in order to do so, they developed another notion of criticality which leads to a different and non-equivalent definition of discrete Morse function. On the other side, first Scoville and Aaronson \cite{Scoville}, then Tanaka \cite{Tanaka_LS_thm}, and afterwards Knudson and Johnson \cite{Knudson}, approached the task by defining another categorical invariant but keeping the usual definition of discrete Morse function. 

The organization of the paper is as follows.  In Section \ref{sec:preliminaries} we recall some definitions and standard results about posets, finite topological spaces and regular complexes. Section \ref{sec:matchings_morse_bott_functions} is devoted to the study of Morse-Bott theory in the context of posets. In Section \ref{sec:fundamental_theorems_and_consequences} we prove the Fundamental Theorems of Morse-Bott Theory in this setting and exploit some of their consequences. Finally, in Section \ref{sec:homological_lusternik-schnirelmann_theorem} we introduce a new notion of homological category and prove the corresponding Lusternik-Schnirelmann theorem for general matchings.


\section{Finite Spaces, Posets and Simplicial Complexes} \label{sec:preliminaries}

This section is devoted to introduce the objects we will work with. Most of the material is well established in the literature, for further details or proofs the reader is referred to \cite{Barmak_book,BarMin12,Bloch,Farmer,SamQuiDaSVil,Minian,Wachs}.

\subsection{Finite spaces and posets}
It is well known  that finite posets and finite $T_0$-spaces are in bijective correspondence.  If $(X, \leq)$ is a poset, a topology $\mathcal{T}$ on $X$ is given by taking the sets $$U_x:=\{y\in X: y\leq x\}$$  as a basis.   On the other hand, if $X$ is a finite $T_0$-space, define for each $x\in X$ the {\it minimal open set} $U_x$ as the intersection of all open sets containing $x$.  Then $X$ may be given an order by defining $y\leq x$ if and only if $U_y\subset U_x$. It is easy to see that  these correspondences are mutual inverses of each other. Moreover a map between posets $f\co X\to Y$ is order preserving if and only if it is continuous when considered as a map between the associated finite spaces.  All posets will be assumed to be finite and by finite space we will mean $T_0$-space. We will use the notion of finite ($T_0$-)space and poset interchangeably. 

We need to introduce some terminology.

\begin{definition}
	A {\it chain} in a poset $X$ is a subset $C\subseteq X$ such that if $x,y\in C$, then either $x\leq y$ or $y\leq x$.     
\end{definition}

\begin{definition}
	The {\it height} of a poset $X$ is the maximum length of the chains in $X$, where the chain $x_0<x_1<\ldots <x_n$ has length $n$. The height $h(x)$ of an element $x\in X$ is the height of $U_x$ with the induced order.
\end{definition}

\begin{definition}
	A poset $X$ is said to be {\it homogeneous} of degree $n$ if all maximal chains in $X$ have length $n$. A poset is {\it graded} if $U_x$ is homogeneous for every $x\in X$. In that case, the {\em degree} of $x$, denoted by $\deg(x)$, is its height.
\end{definition}

We will denote both the height and degree of an element by superscripts, for example $x^{(p)}$.

Let $X$ be a finite poset, $x,y\in X$. If $x< y$ and there is no $z\in X$ such that $x< z< y$, we write $x\prec y$.

For $x\in X$ we also define $\widehat{U}_x:=\{w\in X\colon w<x\}$ as well as $F_x:=\{y\in X : y\geq x\}$ and $\widehat{F}_x:=\{y\in X : y>x\}$.

\subsection{The McCord functors}
We now recall McCord functors between posets and simplicial complexes \cite{McCord}. Given a poset $X$, we define its {\it order complex} $\mathcal{K}(X)$ as the simplicial complex whose $k$-simplices are the non-empty $k$-chains of $X$. Furthermore, given an order preserving map $f\co X\to Y$ between posets, we define the simplicial map $\mathcal{K}(f)\co \mathcal{K}(X) \to \mathcal{K}(Y)$ given by $\mathcal{K}(f)(x)=f(x)$.

Conversely, if $K$ is a simplicial complex, we define the face poset of $K$, $\Delta(K)$, as the poset of simplices of $K$ ordered by inclusion. Given a simplicial map $\phi\co K \to L $ we define the order preserving map $\Delta(\phi)\co \Delta(K) \to \Delta(L)$ given by $\Delta(\phi)(\sigma)=\phi(\sigma)$ for each simplex $\sigma$ of $K$. 

The face poset functor can be defined analogously for regular CW-complexes. That is, given a regular CW-complex $K$, $\Delta(K)$ is the poset of cells of $K$ ordered by inclusion. Given a cellular map $\phi\co K \to L $ we define the order preserving map $\Delta(\phi)\co \Delta(K) \to \Delta(L)$ given by $\Delta(\phi)(\sigma)=\phi(\sigma)$ for each cell $\sigma$ of $K$.

Note that for the simplicial complex $K$, $\mathcal{K}\Delta(K)$ is $\sd(K)$, the first barycentric subdivision of $K$. By analogy, the first subdivision of a finite poset $X$ is defined as $\Delta\mathcal{K}(X)$. 
  
  \begin{theorem}\label{thm:mccords_thms}
  	The following statements hold:
  	\begin{enumerate}
  		\item Let $X$ be a finite $T_0$-space. Then there is a map $\mu_X\co |\mathcal{K}(X)|\to X$ which is a weak homotopy equivalence.
  		\item Let $K$ be a simplicial complex. Then there is a map $\mu_K\co |K|\to \Delta(K)$ which is a weak homotopy equivalence.
  	\end{enumerate}
  \end{theorem}

The maps $\mu_X\co |\mathcal{K}(X)|\to X$ and $\mu_K\co |K|\to \Delta(K)$ will be referred as McCord maps. For details and a proof of the result above see \cite{Barmak_book}.

\subsection{Cellular poset homology} We shall consider a special kind of posets called cellular. They were first introduced by Farmer \cite{Farmer} and then recovered by Minian \cite{Minian}. Farmer's definition is more general while Minian's one  is more adequate for our purposes. That is why we present the latter one.  

\begin{definition}[\cite{Minian}]
	The poset $X$ is {\it cellular}  if it is graded and for every $x\in X$, $\widehat{U}_{x}$ has the homology of a $(p-1)$-sphere, where $p$ is the degree of $x$. 
\end{definition}

Let $X$ be a cellular poset. We denote by $H_*(X)$ the singular homology of $X$. Unless stated otherwise, homology will be considered with integers coefficients. However, the constructions work as well for homology modules with coefficients in any principal ideal domain. We recall the construction due to Farmer \cite{Farmer} and Minian \cite{Minian} of a ``cellular homology theory'' for cellular posets. 

\begin{definition}
	Given a finite graded poset $X$, we define $X^{(p)}$ as the subposet of elements of degree less or equal to $p$, i.e. $$X^{(p)}=\{x\in X\colon \deg(x)\leq p\}.$$ 
\end{definition}

Given the cellular poset $X$, there is a natural filtration by the degree $$X^{(0)}\subset X^{(1)}\subset \cdots X^{(n)}=X$$ which allows to define a {\it cellular chain complex} $(C_*,d)$ as follows: $$C_p(X)=H_p(X^{(p)},X^{(p-1)})=\bigoplus_{\deg(x)=p}H_{p-1}(\widehat{U}_x),$$
which is a free abelian group with one generator for each element of $X$ of degree $p$. The differential $d\colon C_p(X)\to C_{p-1}(X)$ is defined as the composition
$$\begin{tikzcd}
H_p(X^{(p)},X^{(p-1)}) \arrow[r, "\partial"] & H_{p-1}(X^{(p-1)}) \arrow[r, "j"] & H_{p-1}(X^{(p-1)},X^{(p-2)})
\end{tikzcd}$$  
where $j$ is the canonical map induced by the inclusion and $\partial$ is the conecting homomorphism coming from the long exact sequence associated to the pair $(X^{(p)},X^{(p-1)})$. It can be shown \cite{Minian} that the differential 
$$d\colon C_p(X)\to C_{p-1}(X)$$
 can be written as $d(x)=\sum_{w\prec x}\epsilon(x,w)w$ where the incidence number $\epsilon(x,w)$ is the degree of the map 
 $$\widetilde{\partial}\colon \Z =H_{p-1}(\widehat{U}_x)\to H_{p-2}(\widehat{U}_w)=\Z.$$
which coincides with the connecting morphism of the Mayer-Vietoris sequence associated to the covering $\widehat{U}_x=(\widehat{U}_x-\{w\})\cup U_w$ \cite{Minian}.

\begin{theorem}[{\cite[Theorem 3.7]{Minian}}] \label{thm:cellular_homology_isomorphis_homology_cellular_posets}
	Let $X$ be a cellular poset, then $$H_*(C_*(X))\cong H_*(X).$$
\end{theorem}

\subsection{Homologically admissible posets}

We recall the notion of homologically admissible posets introduced by Minian \cite{Minian}. We denote by $\mathcal{H}(X)$ the Hasse diagram associated to the poset $X$.

\begin{definition}[\cite{Minian}]
	Let $X$ be a poset. An edge $(w,x)\in \mathcal{H}(X)$ is homologically admissible if $\widehat{U}_x-\{w\}$ is acyclic. A poset is {\it homologically admissible} if all its edges are homologically admissible.
\end{definition}

The importance of homologically admissible posets, lies, partially, in the following result.

\begin{lema}[{\cite[Remark 3.9]{Minian}}] \label{lema:homolog_admissible_incidence_numbers}
	If $(w,x)$ is a homologically admissible edge of a cellular poset $X$, then the incidence number $\epsilon(x,w)$ is $1$ or $-1$.
\end{lema}

\begin{remark}
	The face posets of regular CW-complexes are homologically admissible \cite[Remark 2.6]{Minian}. However, not every homologically admissible poset is the face poset of a regular CW-complex \cite[Example 2.7]{Minian}.
\end{remark}

\begin{lemma}[\cite{Minian}] \label{lemma:homologiaclly_admissible_is_cellular}
	Let $X$ be a poset. If $X$ is homologically admissible, then it is cellular.
\end{lemma}

\begin{remark}
	In Lemma \ref{lemma:homologiaclly_admissible_is_cellular} it is assumed that the empty set is not acyclic. 
\end{remark}

\subsection{Euler Characteristic}

\begin{definition}
	Let $X$ be a finite graded poset of degree $n$. Denote by $X^{(=p)}$ the elements of degree $p$. The {\em graded Euler-Poincar\'{e} characteristic} of $X$ is defined as the number: $$\chi_{g}(X)=\sum_{i=0}^n(-1)^{i}\#X^{(=p)}.$$
\end{definition}

It is clear that given a poset $X$ of the form $X=\Delta(K)$ for a finite simplicial complex $K$, then $\chi_{g}(X)=\chi(\mathcal{K}(X))$. Moreover, as a consequence of Minian's result (Theorem \ref{thm:cellular_homology_isomorphis_homology_cellular_posets}), the standard homological argument (see for example \cite[p. 146-147]{Hatcher}) proves that for a finite cellular poset $X$, $\chi_{g}(X)=\chi(\mathcal{K}(X))$. However, this does not hold in general for finite posets as the following example illustrates: 

\begin{example}
	Consider the poset $X$ represented in Figure \ref{fig:non_coincidence_of_euler_characs}. Due to the homotopic invariance of $\chi$, $\chi(\mathcal{K}(X))=1$ because $X$ is contractible by removing beat points. However, $\chi_{g}(X)=2$.
	\begin{figure}[htbp]
		\centering
		\includegraphics[scale=0.5]{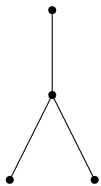}
		\caption{A poset where $\chi_{g}(X)\neq\chi(\mathcal{K}(X))$.}
		\label{fig:non_coincidence_of_euler_characs}
	\end{figure}
\end{example}

\section{Dynamics and Morse-Bott functions for posets}\label{sec:matchings_morse_bott_functions}

In this section we generalize the notion of discrete Morse-Bott function, introduced by Forman \cite{Forman3}, to the context of posets. This also generalizes that of Morse functions on posets defined by Minian \cite{Minian}.

\subsection{Morse functions}
We recall the definition of Morse function for posets introduced by Minian \cite{Minian}. 	

\begin{definition} Let $X$ be a finite poset.  A  {\it Morse function} is a function $f\colon X \to \R$  such that, for every $x\in X$, we have 
	$$
	\#\{y\in X\colon x\prec y \text{ and } f(x)\geq f(y)\}\leq 1
	$$
	and
	$$
	\#\{z\in X \colon z\prec x \text{ and } f(z)\geq f(x)\}\leq 1.
	$$
	
	If $f$ is a Morse function, an element $x\in X$ is said to be {\it critical} if
	$$
	\#\{y\in X \colon x\prec y \text{ and } f(x)\geq f(y)\}=0
	$$
	and
	$$
	\#\{z\in X \colon z\prec x \text{ and } f(z)\geq f(x)\}=0.
	$$
\end{definition}

The set of critical points is denoted by $\crit{f}$. The images of the critical points are called {\it critical values} and the real numbers which are not critical are called {\it regular values}. The points which are not critical values are said to be {\it regular points}.

\subsection{Matchings}

Forman \cite{Forman4} introduced combinatorial vector fields. It is easy to see that this notion can be substituted by the concept of  matching introduced to the context of discrete Morse Theory by Chari \cite{Chari}.

\begin{definition}
	A {\it matching}  in a poset $X$ is a subset $\mathcal{M} \subset X\times X$ such that
	\begin{itemize}
		\item
		$(x, y) \in \mathcal{M}$ implies $x\prec y$;
		\item
		each $x \in X$ belongs to at most one element in $\mathcal{M}$.
	\end{itemize}
\end{definition}

Given a poset $X$, let us denote by $\mathcal{H}(X)$ its associated Hasse diagram.  If $\mathcal{M}$ is a matching in $X$, write $\mathcal{H}_{\mathcal{M}}(X)$ for the directed graph obtained from $\mathcal{H}(X)$ by reversing the orientations of the edges which are not in $\mathcal{M}$. Any node of $\mathcal{H}(X)$ not incident with any edge of $\mathcal{M}$ is called {\it critical}. The set of all critical nodes of $\mathcal{M}$ is denoted by $C_{\mathcal{M}}$.

\begin{definition}
	Let $\mathcal{M}$ be a matching on a poset $X$ and let $x^{(p)}$ and $\tilde{x}^{(p)}$ be two elements of $X$. An $\mathcal{M}$-path, $\gamma$, of index $p$ from $x^{(p)}$ to $\tilde{x}^{(p)}$ is a sequence:
	$$\gamma \co x=x_0^{(p)}\prec y_0^{(p+1)} \succ x_1^{(p)}\prec y_1^{(p+1)} \succ \cdots \prec y_{r-1}^{(p+1)} \succ x_r^{(p)}=\tilde{x}^{(p)}$$
	such that for each $i=0,1,\ldots, r-1$ with $r\geq 1$:
	\begin{enumerate}
		\item $(x_i,y_i) \in \mathcal{M}$,
		\item $x_i\neq x_{i+1}$.
	\end{enumerate}
\end{definition}

A {\it $\mathcal{M}$-cycle} $\gamma$ in $\mathcal{H}_{\mathcal{M}}(X)$ is a closed $\mathcal{M}$-path in $\mathcal{H}_{\mathcal{M}}(X)$ seen as a directed graph. And the matching $\mathcal{M}$ is said to be a {\it Morse matching} if $\mathcal{H}_{\mathcal{M}}(X)$ is acyclic.

\subsection{Critical subposets}\label{subsec:critical_subposets} In this subsection we develop the notion of critical subposet ({\it chain recurrent set}) by means of matchings generalizing the analogous notion introduced by Forman \cite{Forman3} in the context of discrete Morse Theory.  
\begin{definition}
	Let $\mathcal{M}$ be a matching on $X$. We say that $x^{(p)}\in X$ is an element of the {\it chain recurrent set} $\mathcal{R}$ if one of the following conditions holds:
	\begin{itemize}
		\item $x$ is a critical point of $\mathcal{M}$.
		\item There is a $\mathcal{M}$-cycle $\gamma$ in $\mathcal{H}_{\mathcal{M}}(X)$  such that $x\in \gamma$.
	\end{itemize}
\end{definition}

The chain recurrent set decomposes into disjoint subsets $\Lambda_i$ by means of the equivalence relation defined as follows: 
\begin{enumerate}
	\item If $x$ is a critical point, then it is only related to itself.
	\item Given $x,y\in \mathcal{R}$, $x\neq y$,  $x \sim y$ if there is cycle $\gamma$ such that $x,y\in \gamma$. 
\end{enumerate}

Let $\Lambda_1,\ldots,\Lambda_k$ be the equivalence classes of $\mathcal{R}$. The $\Lambda_i's$ are called {\it basic sets}. Each $\Lambda_i$ consists of either a single critical point of $\mathcal{M}$ or a union of cycles.


\begin{example}
	Consider the finite model of $\mathbb{R}P^2$ depicted in Figure \ref{fig:RP2} (see \cite[Example 7.1.1]{Barmak_book}). There is a critical point which is also a basic set, depicted with a cross. Moreover, the dashed and dotted arrows represent another two basic sets, each consisting of one cycle.
	\begin{figure}[htbp]
		\centering
		\includegraphics[scale=0.5]{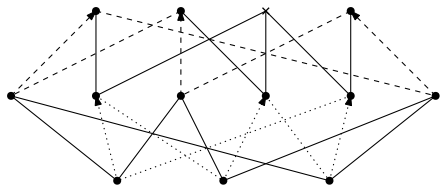}
		\caption{A finite model of $\mathbb{R}P^2$.}
		\label{fig:RP2}
	\end{figure}

\end{example}

\subsection{Integration of matchings}\label{subsection:morse-bott_and_matchings} When working on the differentiable category, Morse theory generalizes naturally to Morse-Bott Theory. The purpose of this subsection is to generalize Minian's integration result for matchings \cite[Lemma 3.12]{Minian} to the context of Morse-Bott functions and arbitrary matchings. 

\begin{definition}
	Given a matching on a finite poset $X$ a function $f\co X\to \R$ is said to be a {\it Morse-Bott} or {\it Lyapunov} function if it is constant on each basic set and it is a Morse function away from the chain recurrent set.
\end{definition}

We say that the {\em critical values} of a Morse-Bott function are the images of the basic sets. The ideas of Forman's proof of \cite[Theorem 2.4]{Forman3} generalize to the context of graded posets giving:

\begin{theorem}[Integration of matchings]\label{thm:integration_matchings}
	Let $X$ be a finite graded  poset and let $\mathcal{M}$ be a matching in $X$. Then there exists a Morse-Bott function $f\co X\to \R$ such that:
	\begin{enumerate}
		\item If $x^{(p)}\notin \mathcal{R}$ and $y^{(p+1)}\succ x$, then
		$$\begin{cases} 
		f(x)<f(y), & \text{ if } (x,y) \notin \mathcal{M}, \\
		f(x)\geq f(y), & \text{ if } (x,y) \in \mathcal{M}.
		\end{cases}$$
		\item If $x^{(p)} \in \mathcal{R}$ and $y^{(p+1)}\succ x$, then
		$$\begin{cases} 
		f(x)=f(y), & \text{ if } x\sim y, \\
		f(x)< f(y), & \text{ if } x \nsim y.
		\end{cases}$$
	\end{enumerate}
\end{theorem}

We introduce the following definition: given a finite poset $X$ and a Morse-Bott function $f\co X\to \R$, for each $a\in \mathbb{R}$ we write
$$X_a=\bigcup\limits_{f(x)\leq a}{U_x}.$$

\subsection{Morse-Smale matchings}

In this subsection we generalize the notion of Morse-Smale vector field from the context of simplicial complexes \cite{Forman3} to the setting of finite spaces.

Let $X$ be a homologically admissible poset and let $\mathcal{M}$ be a matching on $X$. A $\mathcal{M}$-cycle $\gamma$ is {\em prime} if they do not exist a natural number $n>1$ and a $\mathcal{M}$-cycle $\widetilde{\gamma}$ such that   $\gamma$ is the concatenation of $\widetilde{\gamma}$ $n$ times (see \cite[Definition 5.3]{Forman3} for details).

An equivalence relation on the set of $\mathcal{M}$-cycles is defined as follows. Two $\mathcal{M}$-cycles $\gamma$ and $\widetilde{\gamma}$ are equivalent if $\widetilde{\gamma}$ is the result of varying the starting point of $\gamma$ (see \cite[p. 631]{Forman3} for an example). An equivalence class of $\mathcal{M}$-cycles is called a {\em closed $\mathcal{M}$-orbit}. The equivalence class of $\gamma$ is denoted by $[\gamma]$. The concepts of {\em prime closed $\mathcal{M}$-orbit} and {\em index of an closed $\mathcal{M}$-orbit} are defined as expected (see \cite{Forman3} for details).

A special kind of matching which will play an important role is the following. In a certain sense, we control the complexity of the chain recurrent set.

\begin{definition}
	Let $X$ be a homologically admissible poset. A matching $\mathcal{M}$ on $X$ is a {\em Morse-Smale matching} if the chain recurrent set $\mathcal{R}$ consists only of critical points and pairwise disjoint prime closed $\mathcal{M}$-orbits.
\end{definition}

\section{Fundamental Theorems and consequences}\label{sec:fundamental_theorems_and_consequences}

 The purpose of this Section is to prove the Fundamental Theorems of Morse Theory for  Morse-Bott functions on posets and obtain some consequences.
 
\subsection{Fundamental Theorems} 
  In what follows, we extend the equivalence relation  defined in Subsection \ref{subsec:critical_subposets} from $\mathcal{R}$ to all $X$ by saying that a point which is not critical is an equivalence class on its own. 

\begin{definition}
	Given a finite poset $X$, $x\in X$ and a matching $\mathcal{M}$ on $X$, we define: $$\partial [x]=\{w\in X \co w\prec \tilde{x} \text{ for some } \tilde{x}\sim x \text{ but } w \nsim \tilde{x}\}.$$
\end{definition}

\begin{example}
	Consider the poset depicted in Figure \ref{fig:RP2}. In Figure \ref{fig:RP2_d[x]} we show $\partial [x]$ for any $x$ in the dashed cycle of Figure \ref{fig:RP2}.
	\begin{figure}[htbp]
		\centering
		\includegraphics[scale=0.5]{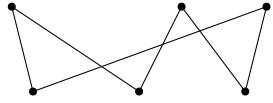}
		\caption{Example of $\partial [x]$.}
		\label{fig:RP2_d[x]}
	\end{figure}
\end{example}
 We introduce some auxiliary notation. For each edge $(x,y)\in \mathcal{M}$, we say that $x$ is the {\em source} of the edge and $y$ is the {\em target}.  For convenience, we define the {\em source} and {\em target maps} (only defined for elements in the matching $\mathcal{M}$) as follows: given $(x,y)\in \mathcal{M}$, $s(y)=x$ and $t(x)=y$.

The lemma below follows from the definition of matching:

\begin{lema}\label{lema:technical_Lyapunov}
	Let $\gamma$ be a cycle of index $p$ and let $u^{(p-1)}\in X$, $\tilde{v}^{(p)}\in X$, $w^{(p+1)}\in X$ and $r^{(p+2)}\in X$ such that $u,\tilde{v},w,r\notin \gamma$. Then it holds the following: $$t(u)\notin \gamma, \; t(\tilde{v})\notin \gamma, \; s(w)\notin \gamma \; \text{ and } \; s(r)\notin \gamma.$$
\end{lema}

Our next result is a homological collapsing theorem for Morse-Bott functions. As a consequence of the Lemma \ref{lema:technical_Lyapunov}, the elements of a cycle can not be connected by arrows with elements which are not in the cycle. Therefore, the result below follows from \cite[Theorem 4.2.2]{SamQuiDaSVil}.

\begin{theorem}\label{thm:homological_collapse_thm_Lyapunov} Let $X$ be a finite homologically admissible poset and let $f\colon X \to \mathbb{R}$ be a Morse-Bott function. If $[a,b]$ contains no critical values,  then $i\colon X_a \hookrightarrow X_b$ induces an isomorphism in homology.
\end{theorem}

In this generalized context, we also have a result which explains what happens when a critical value is reached.  

\begin{theorem}\label{thm:homot_poset_crit_points_Lyapunov}  
	Let $X$ be a finite homologically admissible poset and let $f\co X \to \R$ be a Morse-Bott function. If $f(x)\in [a,b]$ is a critical value and there are no other values of $f$ in $[a,b]$, then $X_b=X_a \cup_{\partial [x]}[x]$. 
\end{theorem}

\begin{proof}
	There are two cases to consider. First, assume that $[x]$ is a critical point, then the results reduces to \cite[Theorem 4.2.8]{SamQuiDaSVil}. So, assume $[x]$ is a cycle of index $p$. Let $\tilde{f}\co X/\sim \to \R$ denote the function induced by $f$ on the set of equivalence classes. We may assume that $\tilde{f}$ is injective, that $\tilde{f}([x])>a$ and that the only critical subposet in $f^{-1}([a,b])$ is $[x]$. 
	
	Since $[x]$ is a cycle and $f(x)$ is a critical value, then given $y^{(p+1)}\succ \tilde{x}$ and $y\notin [x]$ with $\tilde{x}\in [x]$, $f(y)>f(\tilde{x})$. Hence, $f(y)>b$ and Lemma \ref{lema:technical_Lyapunov} guarantees that $f(z)>b$ for every $z>\tilde{x}$, $z\notin [x]$. Therefore, $[x]\cap X_a=\emptyset$. Given any $w^{(p-1)} \prec \tilde{x}^{(p)}$, $\tilde{x}\in [x]$ and $w\notin [x]$ or $w^{(p)} \prec \tilde{x}^{(p+1)}$, $\tilde{x}\in [x]$ and $w\notin [x]$, due to the criticality of $[x]$, it holds that $f(w)<f(\tilde{x})$. Therefore $f(w)<a$ and $w\in X_a$. Hence $\partial [x] \subset X_a$. That is, $X_b=X_a \cup_{\partial [x]}[x]$. 
\end{proof}

\subsection{Morse-Bott inequalities}\label{subsubsection:morse_bott_pitcher_inequalities}
In this subsection we generalize Morse-Bott inequalities from the context of CW-complexes \cite[Theorem 3.1]{Forman3} to the setting of posets. This result can be seen  a combinatorial analogue of a theorem due to Conley \cite[Theorem 1.2]{Franks} \cite{Conley}.  Again, we assume that our coefficients are any principal ideal domain $R$. From now on the poset $X$ is assumed to be homologically admissible.

Given a subposet $Y\subset X$ we denote by $\bar{Y}$ the subposet $\cup_{x\in Y}U_x$ and by $\dot{Y}=\bar{Y}-Y$.

\begin{definition}
For each $k\geq 0$, we define $$m_k=\sum_{\text{basic sets } \Lambda_i} \rank H_k(\bar{\Lambda}_i,\dot{\Lambda}_i).$$
\end{definition}

Observe that in the particular case we have a Morse matching, then the basic sets are just critical points and $m_k$ is the number of critical points of index $k$.

\begin{lema}
	If the index of the basic set $\Lambda_i$ is $p$, then $H_k(\bar{\Lambda}_i,\dot{\Lambda}_i)=0$ unless $k=p,p+1$. Moreover, if $\Lambda_i$ is just a critical point $x^{(p)}$, then $H_k(\bar{\Lambda}_i,\dot{\Lambda}_i)=0$ for $k\neq p$ and the principal domain of coefficients, $R$, for $k=p$.
\end{lema}

\begin{proof}
	For convenience, during the proof we will denote $\Lambda_i=\Lambda$. Since all the posets involved are cellular we can use cellular homology. Consider the Homology Long Exact Sequence for the pair $(\bar{\Lambda},\dot{\Lambda})$:
	$$
	\begin{tikzcd}[cramped]
	\cdots \arrow[r, ""] & H_{p}(\bar{\Lambda})    \arrow[r, ""] & H_{p}(\bar{\Lambda},\dot{\Lambda}) \arrow[r, ""] & H_{p-1}(\dot{\Lambda}) 
	\ar[out=-15, in=150]{dll}\\
	& H_{p-1}(\bar{\Lambda})    \arrow[r, "j"] & H_{p-1}(\bar{\Lambda},\dot{\Lambda}) \arrow[r, "\partial"] & H_{p-2}(\dot{\Lambda})  
	\ar[out=-15, in=150]{dll}{\cong}\\
	& H_{p-2}(\bar{\Lambda})    \arrow[r, ""] & H_{p-2}(\bar{\Lambda},\dot{\Lambda}) \arrow[r, ""] & H_{p-3}(\dot{\Lambda}) \arrow[r, ""] & \cdots \\
	\end{tikzcd}
	$$
	First of all, the homomorphism $H_{k}(\dot{\Lambda})\to H_{k}(\bar{\Lambda})$ is an isomorphism for $k\leq p-2$, so  $H_{k}(\bar{\Lambda},\dot{\Lambda})=0$ for $k\leq p-2$. Second, we have that: $$H_{p-1}(\bar{\Lambda},\dot{\Lambda})=\ker \partial=\Imm j\cong \frac{H_{p-1}(\bar{\Lambda})}{\ker j}\cong \frac{H_{p-1}(\bar{\Lambda})}{\Imm i}.$$
	Third, the homomorphism $H_{p-1}(\dot{\Lambda}) \to H_{p-1}(\bar{\Lambda})$ induced by the inclusion is surjective by the construction of cellular homology. Therefore $H_{p-1}(\bar{\Lambda},\dot{\Lambda})=0$. Fourth, if $\Lambda$ is just a critical point $x^{(p)}$, then $H_k(\bar{\Lambda},\dot{\Lambda})=H_k(U_x,\widehat{U}_x)$ and by cellularity of $X$ and the Homology Long Exact Sequence for the pair $(U_x,\widehat{U}_x)$ the result follows.
\end{proof}

We denote by $b_k$ the Betti number of dimension $k$ with coefficients in the principal domain $R$.

Taking into account the ideas involved in the proof of \cite[Theorem 3.1]{Forman3}  and our Theorems \ref{thm:homological_collapse_thm_Lyapunov} and \ref{thm:homot_poset_crit_points_Lyapunov} yields the Strong Morse-Bott inequalities:

\begin{theorem}[Strong Morse-Bott inequalities]\label{thm_strong_morse__bott_inequalities}
	Let $X$ be a homologically admissible poset and let $\mathcal{M}$ be a matching on $X$. Then, for every $k\geq 0$:
	$$m_k-m_{k-1}+\cdots +(-1)^k m_0\geq b_k-b_{k-1}+\cdots +(-1)^k b_0.$$
\end{theorem}

From the standard argument (see \cite[p. 30]{Milnor}), we obtain the Weak Morse inequalities:

\begin{corollary}[Weak Morse-Bott inequalities] \label{coro:weak_morse_ineq}
	Let $X$ be a homologically admissible poset and let $\mathcal{M}$ be a matching on $X$. Then: \begin{enumerate}
		\item For every $k \geq 0$, $m_k\geq b_k$.
		\item $\chi(X)=\sum_{i=0}^{\deg(X)}(-1)^k b_k=\sum_{i=0}^{\deg(X)}(-1)^k m_k.$
	\end{enumerate}
\end{corollary}

\subsection{Morse-Smale matchings}\label{subsec:Morse_Smale}
	
In this section we generalise \cite[Section 7]{Forman2} to the context of homologically admissible posets while improving some of the results even in the case of simplicial or regular CW-complexes.

Let $X$ be a homologically admissible poset and let $\mathcal{M}$ be a Morse-Smale matching on $X$. We denote by $c_k$ the number os critical points of index $k$ and by $A_k$ the number of prime closed $\mathcal{M}$-orbits of index $k$. Denote by $\mu_k$ the minimum number of generators of the torsion subgroup $T_k$ of $H_k(X)$.

Combining the proof of \cite[Theorem 7.1]{Forman2} with our Pitcher strengthening of Morse inequalities \cite[Corollary 5.2.3]{SamQuiDaSVil} we obtain the following improvement of \cite[Theorem 7.1]{Forman2}, taking torsion into account.

\begin{theorem}\label{thm:morse_bott_inequalities_number_orbits}
	Let $X$ be a homologically admissible poset and let $\mathcal{M}$ be a Morse-Smale matching on $X$. Let the coefficients $R$ be a principal ideal domain. Then, for every $k\geq 0$:
	$$A_k+\sum_{i=0}^k (-1)^{i}c_{k-i} \geq \mu_k + \sum_{i=0}^k (-1)^{i}b_{k-i}.$$
\end{theorem}

\begin{definition}
	Let $X$ be a homologically admissible poset and let $\mathcal{M}$ be a Morse-Smale matching on $X$. Endow each element of $X$ with an orientation.  Let $\gamma$ be an $\mathcal{M}$-path
	$$\gamma \co x_0^{(p)}\prec y_0^{(p+1)} \succ x_1^{(p)}\prec y_1^{(p+1)} \succ \cdots \prec y_{r-1}^{(p+1)} \succ x_r^{(p)}.$$ 
	We define the multiplicity of $\gamma$ by $$\prod_{i=0}^{r-1}-\langle d_{(p+1)}y_i,x_i\rangle_p \langle d_{(p+1)}y_i,x_{i+1}\rangle_p$$
	where $d$ is the cellular boundary operator and $\langle \bullet,\bullet \rangle_p$ is the inner product on $C_p(X)$ such that the degree $p$ elements of $X$ are mutually orthogonal.
\end{definition}

\begin{remark}
	Observe that the multiplicity of a path is always $1$ or $-1$ due to Lemma \ref{lema:homolog_admissible_incidence_numbers}.
\end{remark}

\begin{remark}
	The generalization of \cite[Lemma 4.6]{Forman3} to our context is straightforward.
\end{remark}

Both \cite[Theorem 7.3]{Forman3} and \cite[Corollary 7.4]{Forman3} generalise to our setting with the same proofs:

\begin{theorem}\label{thm:morse_bott_inequalities_multiplicity _one_orbits}
	Let $X$ be a homologically admissible poset and let $\mathcal{M}$ be a Morse-Smale matching on $X$. Let the coefficients be the field $\R$. Denote by $A'_p$ the number of closed $\mathcal{M}$-orbits of index $p$ and multiplicity 1. Then, for every $k\geq 0$:
	$$A'_k+\sum_{i=0}^k (-1)^{i}c_{k-i} \geq \sum_{i=0}^k (-1)^{i}b_{k-i}(\R).$$
\end{theorem}
	
\begin{remark}
	While \cite[Corollary 7.4]{Forman3} refined \cite[Theorem 7.2]{Forman3}, Theorem \ref{thm:morse_bott_inequalities_multiplicity _one_orbits} does not refine our improved Theorem \ref{thm:morse_bott_inequalities_number_orbits}. They are complementary results. 
\end{remark}

\section{Homological Lusternik-Schnirelmann Theorem}\label{sec:homological_lusternik-schnirelmann_theorem}

The purpose of this section is to prove a Lusternik-Schnirelmann Theorem for general matchings and a suitable definition of homological category. 

\subsection{Definition of the homological chain category and first properties}
Let $(C_*,\partial)$ denote a free chain complex of abelian groups such that each term $C_p$ is finitely generated and only finitely many of the $C_p$ are non zero. We define the {\em rank} of $C_*$ as $\rank(C_*)=\sum_p\rank(C_p)$.

\begin{definition}\label{def:homological_chain_category_for_chain_complexes}
	Let $(C_*,\partial)$ be a free chain complex of abelian groups. We define its {\em homological chain  category}
	\begin{equation*}
		\HCcat(C_*)=\inf\Bigg\{
		\begin{aligned}
			\rank(B_*)\co  B_* \text{ bounded subcomplex of } C_* \text{ and the } &\\ \text{ inclusion } i\co B_* \hookrightarrow C_* \text{ is a quasi-isomorphism.} 
		\end{aligned} \Bigg\}
	\end{equation*}
	
\end{definition}

Let $X$ be a topological space. For all the definitions that follow we consider coefficients in $\Z$. We denote by $S_*(X)$ its singular chain complex.

\begin{definition}
	Let $X$ be a topological space. We define its {\em homological chain  category} $\HCcat(X)=\HCcat(S_*(X))$.
\end{definition}

 We introduce a homological lower bound for $\HCcat(X)$ analogous to the Pitcher strengthening of Morse inequalities.

\begin{prop}\label{prop:homological_lower_bound_HCcat}
	Let $X$ be a topological space with finitely generated homology. Then $$\sum_k b_k + 2\sum_k \mu_k \leq \HCcat(X).$$
\end{prop}

\begin{proof}
Let us denote by $(B_*,\partial)$ a bounded chain  complex whose homology is isomorphic to $H_*(X)$. By standard algebra (see, for example \cite[Theorem 4.11]{Prasolov}), we have $b_k+\mu_k+\mu_{k-1}\leq \rank(B_k)$. Now the result follows by a sum indexed by the dimension.
\end{proof}
	
\begin{cor}
Let $X$ be a homologically admissible poset or a CW-complex with finitely generated homology. Then $$\chi(X) \leq \HCcat(X).$$
\end{cor}

In fact, the bound given by Proposition \ref{prop:homological_lower_bound_HCcat} is the best possible as a consequence of the following result due to Pitcher \cite[Lemma 13.2]{Pitcher}.

\begin{prop}
	Let $(C_*,\partial)$ be a free chain complex with singular homology groups $H_k(X)$, $k=0,1,\ldots$ Denote by $\mu_k$ the minimum number of generators of the torsion subgroup $T_k$ of $H_k(X)$ and denote by $b_k$ the rank of $H_k(X)$. Then there exists a free chain complex $(L,\partial^{L})$ such that:
	\begin{enumerate}
		\item For every $k\geq 0$, the group $L_k$ has rank $b_k+\mu_k+\mu_{k-1}$.
		\item There exists a monomorphism $i\co L\hookrightarrow C$ which is a chain map.
		\item The monomorphism $i\co L\hookrightarrow C$ is a quasi-isomorphism.
	\end{enumerate}  
\end{prop}
 
\begin{corollary}\label{coro:exact_value_hccat}
	Let $X$ be a topological space with finitely generated homology. Then $$\HCcat(X)=\sum_k b_k + 2\sum_k \mu_k.$$
\end{corollary}

Moreover, observe that a topological $X$ is acyclic if and only if $\HCcat(X)=1$.

%

As a consequence of \cite[Example 1.33]{CLOT} we have the following result relating the homological chain category to the Lusternik-Schnirelmann category:

\begin{prop}\label{prop:relation_cat_HCcat}
	Let $K$ be a simply connected CW-complex with finitely generated homology groups such that there exists $n$ satisfying $H_n(K)\neq 0$ and $H_p(K)=0$ for $p>n$.  Then
	$$\cat(K)\leq \HCcat(K).$$
\end{prop}

The result does not necessarily hold if we remove the simply connectedness hypothesis, as the following example shows:

\begin{example}
	Consider the Poincar\'{e} homology 3-sphere $M$. Observe that $\HCcat(M)=\HCcat(\mathbb{S}^3)=2$. However, $\cat(M)\geq 3$ \cite{Fox}.
\end{example}

\subsection{Homological Lusternik-Schnirelmann Theorem}

In this subsection we state and prove a Lusternik-Schnirelmann Theorem for the homological chain category and general matchings on posets.

\begin{theorem}\label{thm:LS_theorem_premium_version}
	Let $X$ be a homologically admissible poset and let $\mathcal{M}$ be a Morse-Smale matching on $X$. Then 
	$$\HCcat(X)\leq \sum_{{\mathrm{basic \: sets} \: } \Lambda_i} \HCcat(\Lambda_i).$$
	In particular, given Morse matching $\mathcal{M}$ on $X$, then $\HCcat(X)$ is a lower bound for the number of critical elements of $\mathcal{M}$.
\end{theorem} 

\begin{proof}
	We will define a Morse matching $\mathcal{M}^{*}$ by means of perturbing $\mathcal{M}$. The idea is to replace each prime closed orbit by two critical points. This will be achieved by removing exactly one of the edges of the matching in each closed orbit. By repeating the technique used in the proof of \cite[Theorem 7.1]{Forman3}, we obtain a Morse matching $\mathcal{M}^{*}$ satisfying $m_p^{*}=c_p+A_p+A_{p-1}$, where $m_p^{*}$ denotes the number of critical points of index $p$ of the matching $\mathcal{M}^{*}$ (see Subsection \ref{subsec:Morse_Smale} for the definition of $A_p$). 
	
	Recall that $C_*(X)$ denotes the cellular chain complex of $X$. We define a map $V\co C_p(X)\to C_{p+1}(X)$ as follows:
	$$V(x)= \begin{cases} 
	-\epsilon(y,x)y, & \text{if there exists } y\in X \text{ with } (x,y)\in \mathcal{M}^{*},\\
	0, & \text{otherwise.} 
	\end{cases}$$

	Following the ideas of Minian \cite{Minian}, define the discrete flow operator $\phi \co C_p(X)\to C_{p}(X)$ as $\phi=\Id+dV+Vd$. The $\phi$-invariant chains 
	$$C_{p}^{\phi}(X)=\{c\in C_p(X)\co \phi(c)=c\}$$ form a well-defined subcomplex of $(C_{*}(X),d)$ \cite{Minian}. Moreover, the inclusion of  $(C_{*}^{\phi}(X),d)$ into $(C_{*}(X),d)$ induces isomorphisms in homology  and $C_{p}^{\phi}(X)$ is isomorphic to the free abelian group spanned by the critical $p$-elements of $X$ \cite{Minian}. As a consequence: \begin{equation}\label{eq:HCcat_sum_preliminar}
	\HCcat(C_*(X))\leq \sum_p m_p^{*}=\sum_p c_p+A_p+A_{p-1}.
	\end{equation}
	
	There are two kinds of basic sets for $\mathcal{M}$: critical points and disjoint closed $\mathcal{M}$-orbits. Observe that if $\Lambda_i$ is a critical point, then $\HCcat(\Lambda_i)=1$ while if  $\Lambda_i$ is a closed orbit, then $\HCcat(\Lambda_i)= 2$. So, from Equation (\ref{eq:HCcat_sum_preliminar}), it follows that:
	$$\HCcat(C_*(X))\leq \sum_{ \text{basic sets } \Lambda_i} \HCcat(\Lambda_i).$$
	Finally, observe that $\HCcat(X)=\HCcat(C_*(X))$ due to the isomorphism between cellular homology and singular homology for cellular posets (Theorem \ref{thm:cellular_homology_isomorphis_homology_cellular_posets}). 
\end{proof}
	
\begin{remark}
	In the proof of Theorem \ref{thm:LS_theorem_premium_version}, Equation (\ref{eq:HCcat_sum_preliminar}) could also be derived as a consequence of combining our Pitcher strengthening of Morse-inequalities \cite[Corollary 5.2.3]{SamQuiDaSVil} applied to the matching $\mathcal{M}^{*}$ with Corollary \ref{coro:exact_value_hccat}.
\end{remark}	

As a consequence of \cite[Theorem 3.3.6]{SamQuiDaSVil}, we obtain the following corollary:

\begin{cor}\label{coro:LS_theorem_free_version}
	Let $X$ be a homologically admissible poset and let $f\co X\to \R$ be a Morse function. Then $\HCcat(X)$ is a lower bound for the number of critical points of $f$.
\end{cor}
	
\begin{remark}
	Let $K$ be a simplicial complex or, more generally, a regular CW-complex $K$. Recall that its face poset $\Delta(K)$ is a homologically admissible poset. Moreover, the chain complex $C_\bullet(\Delta(K),d)$ where $d$ is the cellular boundary operator coincides with the chain complex $C_\bullet(K,\partial)$ where $\partial$ is the cellular -or simplicial in case $K$ is a simplicial complex- boundary operator. Therefore $\HCcat(\Delta(K))=\HCcat(K)$. Hence, we have in particular a simplicial homological Lusternik-Schnirelmann Theorem.
\end{remark}

%
%

\bibliographystyle{abbrvnat}
\bibliography{biblio}

\end{document}